\documentclass[8pt]{amsart}

\usepackage{amsmath}
\usepackage{amssymb}
\usepackage{amsthm}
\usepackage{enumerate}
\usepackage[all,arc,curve,color,frame]{xy}
\usepackage[utf8]{inputenc}

\usepackage{color, graphicx}
\usepackage[square,sort,comma,numbers]{natbib}

\renewcommand\section{\@startsection{section}{1}{\z@}
                                   {-3.5ex \@plus -1ex \@minus -.2ex}
                                   {2.3ex \@plus .2ex}
                                   {\normalfont\large\bfseries}}
\renewcommand\subsection{\@startsection{subsection}{2}{\z@}
                                   {-3.25ex\@plus -1ex \@minus -.2ex}
                                   {1.5ex \@plus .2ex}
                                   {\normalfont\normalsize\bfseries}}
\renewcommand\subsubsection{\@startsection{subsubsection}{3}{\z@}
                                   {-3.25ex\@plus -1ex \@minus -.2ex}
                                   {1.5ex \@plus .2ex}
                                   {\normalfont\normalsize\bfseries}}
\renewcommand\paragraph{\@startsection{paragraph}{4}{\z@}
                                   {3.25ex \@plus1ex \@minus.2ex}
                                   {-1em}
                                   {\normalfont\normalsize\bfseries}}

\makeatother

\newdimen\tableauside\tableauside=1.0ex
\newdimen\tableaurule\tableaurule=0.4pt
\newdimen\tableaustep
\def\phantomhrule#1{\hbox{\vbox to0pt{\hrule height\tableaurule
width#1\vss}}}
\def\phantomvrule#1{\vbox{\hbox to0pt{\vrule width\tableaurule
height#1\hss}}}
\def\sqr{\vbox{%
  \phantomhrule\tableaustep

\hbox{\phantomvrule\tableaustep\kern\tableaustep\phantomvrule\tableaustep}%
  \hbox{\vbox{\phantomhrule\tableauside}\kern-\tableaurule}}}
\def\squares#1{\hbox{\count0=#1\noindent\loop\sqr
  \advance\count0 by-1 \ifnum\count0>0\repeat}}
\def\tableau#1{\vcenter{\offinterlineskip
  \tableaustep=\tableauside\advance\tableaustep by-\tableaurule
  \kern\normallineskip\hbox
    {\kern\normallineskip\vbox
      {\gettableau#1 0 }%
     \kern\normallineskip\kern\tableaurule}%
  \kern\normallineskip\kern\tableaurule}}
\def\gettableau#1 {\ifnum#1=0\let\next=\null\else
  \squares{#1}\let\next=\gettableau\fi\next}

\makeatletter

\renewcommand\section{\@startsection{section}{1}{\z@}
                                   {-3.5ex \@plus -1ex \@minus -.2ex}
                                   {2.3ex \@plus .2ex}
                                   {\normalfont\large\bfseries}}
\renewcommand\subsection{\@startsection{subsection}{2}{\z@}
                                   {-3.25ex\@plus -1ex \@minus -.2ex}
                                   {1.5ex \@plus .2ex}
                                   {\normalfont\normalsize\bfseries}}
\renewcommand\subsubsection{\@startsection{subsubsection}{3}{\z@}
                                   {-3.25ex\@plus -1ex \@minus -.2ex}
                                   {1.5ex \@plus .2ex}
                                   {\normalfont\normalsize\bfseries}}
\renewcommand\paragraph{\@startsection{paragraph}{4}{\z@}
                                   {3.25ex \@plus1ex \@minus.2ex}
                                   {-1em}
                                   {\normalfont\normalsize\bfseries}}

\makeatother

\newcommand{\be}{\begin{equation}}
\newcommand{\ee}{\end{equation}}
\newcommand{\bea}{\begin{eqnarray}}
\newcommand{\eea}{\end{eqnarray}}

\newcommand{\id}{\hbox{1\kern-.27em l}}

\newcommand{\al}{\alpha}

\newcommand{\bet}{\beta}

\newcommand{\la}{\lambda}

\newcommand{\cN}{\mathcal{N}}

\newcommand{\non}{\nonumber}

\newcommand{\SO}{\mathrm{SO}}

\newcommand{\Sp}{\mathrm{Sp}}

\newcommand{\Spin}{\mathrm{Spin}}

\newtheorem{lem}{Lemma}[section]
\newtheorem{prop}{Proposition}[section]

\newtheorem{remark}{Remark}[section]

\newcommand{\ba}{\begin{array}}
\newcommand{\ea}{\end{array}}

\newtheorem{Def}{Definition}

\begin{document}

\title[Invariants of partitions and    representative elements]
{Invariants of partitions and  representative elements }

\author[Bao Shou]{Bao Shou$^{\dag}$}
\author[Qiao Wu]{Qiao Wu$^*$}

\address{$^*$ College of Logistics and E-commerce \\Zhejiang Wanli University\\
No.8 South Qianhu Road, Ningbo 315100, P.R.China}
\address{$\dag$ Center  of Mathematical  Sciences\\Zhejiang University \\
Hangzhou 310027,  China}

\email{$^*$10920005@zju.edu.cn, \quad   $^{\dag}$ bsoul@zju.edu.cn }

\subjclass[2010]{05E10}

\keywords{Partition, fingerprint, Kazhdan-Lusztig map, construction, rigid semisimple operator}

\date{}

\maketitle

\begin{abstract}
The symbol invariant is used to describe the Springer correspondence for the classical  groups by Lusztig. And the fingerprint invariant can be  used to describe the Kazhdan-Lusztig map.   They are   invariants of rigid semisimple operators described by pairs of partitions $(\lambda^{'}, \lambda^{''})$.
We construct a nice representative element of the rigid semisimple operators with the same symbol invariant. The fingerprint of the representative element  can be obtained immediately.  We also discuss the  representative element  of rigid semisimple operator with the same fingerprint invariant. Our construction can be regarded as the maps between these two invariants.

\end{abstract}

\tableofcontents

\section{Introduction}
A partition $\la$ of the positive integer $n$ is a decomposition $\sum_{i=1}^l \la_i = n$  ($\la_1\ge \la_2 \ge \cdots \ge \la_l$), with the length  $l$. There is an   one to one correspondence between the partitions and  Young tableaux.  Young tableaux occur in a number of branches of mathematics and physics. They  are also  tools for the construction of  the eigenstates of Hamiltonian System \cite{Sh11} \cite{{Sh14}} \cite{{Sh15}} and  label  surface operators\cite{GW08}.

Surface operators are two-dimensional defects, which are  natural generalisations of the 't~Hooft operators.  In \cite{GW06},  Gukov and Witten initiated a study of surface operators in $\mathcal{N}=4$ super Yang-Mills theories.   The $S$-duality  \cite{Montonen:1977} assert that
$$S : \; (G, \tau) \rightarrow  ( G^{L}, - 1 / n_{\mathfrak{g}} \tau),$$
where $n_{\mathfrak{g}}$ is 2 for $F_4$,  3 for $G_2$, and 1 for other semisimple classical groups \cite{GW06};  $\tau $ is usual gauge coupling constant and $G^{L}$ is  Langlands dual group of $G$.
In \cite{GW08},  Gukov and Witten extended their earlier analysis  and identified a subclass of surface operators called {\it 'rigid'} surface operators expected to be closed under $S$-duality.
There are two types rigid surface operators:  unipotent and semisimple. The rigid semisimple surface operators  are labelled by pairs of  partitions $(\lambda^{'}, \lambda^{''})$. Unipotent rigid surface operators arise  when one of the  partitions is empty.

In \cite{Wy09}, Wyllard made some explicit proposals for how the $S$-duality maps should act on rigid surface operators  based on the  invariants of partitions.
The invariant {\it fingerprint} based on  Kazhdan-Lusztig map for the classical groups \cite{Lusztig:1984}\cite{Symbol 2}.
 The Kazhdan-Lusztig map  is a map from the unipotent conjugacy classes to the set of conjugacy classes of the Weyl group. It can be extended to the case of rigid semisimple conjugacy classes\cite{Spaltenstein:1992}. The rigid semisimple conjugacy classes and  the conjugacy classes of the Weyl group are described by pairs of partitions $(\lambda^{'};\lambda^{''})$   and pairs of partitions  $[\alpha;\beta]$, respectively. The fingerprint invariant  is  a map between these two classes of objects. There is another invariant  {\it symbol } related to the the Springer correspondence\cite{CM93}.

Compared to the fingerprint, the symbol invariant is much easier to be calculated and  more    convenient   to  find the $S$-duality maps of surface operators.
In \cite{Shou 1:2016}, we find a new subclass of rigid surface operators related by S-duality.  In \cite{Shou-sc}, we found a construction  of  symbol for the rigid partitions in the  $B_n, C_n$, and $D_n$ theories.  Another construction of symbol is  given in \cite{SW17}. We discuss the basic properties of  fingerprint and its constructions in  \cite{SW17-2}.

A problematic mismatch in the total number of rigid surface operators between  the $B_n$ and the $C_n$ theories was pointed out in \cite{GW08} \cite{Wy09}.      In \cite{rso},  we clear up cause of  this discrepancy and  construct all the $B_n/C_n$ rigid surface operators which can not have a dual. The fingerprint invariant is   assumed to be  equivalent to the symbol invariant \cite{Wy09}.
In \cite{iv}, we   prove the symbol invariant of partitions implies the fingerprint invariant of partitions. The constructions of the symbol invariant  and the fingerprint invariant in previous works play a significant role  in the proof.
This problem suggest us to find map between  two invariants.

The following is an outline of  this article.  In Section \ref{rigidr}, we summary   some basic results related to the rigid partition in \cite{Wy09}.
In Section \ref{symbol-s},  we introduce  the definition and the construction of symbol in \cite{Shou-sc}. Then we present the construction of the representative element $(\lambda^{'}, \lambda^{''})_R$ of rigid surface operators with the   given symbol.
In Section \ref{finger},  we give the constructions of  the fingerprint invariant in \cite{SW17-2}. We calculate the fingerprint of  the representative element $(\lambda^{'}, \lambda^{''})_R$.  We also discuss  the construction of the representative element $\mu_r$  of $\lambda^{'}+ \lambda^{''}$ with the    fingerprint given. our construction can be regarded as the maps between  the fingerprint  invariant and symbol invariant.
In Section \ref{cd},  we discuss the realization of the above results for the $C_n$ and $D_n$ theories.

We will focus on theories with gauge groups $\SO(2n)$  which are Langlands self-duality and the gauge groups $\Sp(2n) $ whose Langlands dual group is $\SO(2n+1)$.

\section{Rigid Partitions in the  $B_n, C_n$, and $D_n$ theories}\label{rigidr}
In this section, we introduce the rigid partitions in the $B_n$, $C_n$, and $D_n$ theories.
For the $B_n$($D_n$)theories, unipotent conjugacy  classes  are in one-to-one correspondence with partitions of $2n{+}1$($2n$) where all even integers appear an even number of times.  For the $C_n$ theories, unipotent conjugacy  classes  are in one-to-one correspondence with partitions $2n$ for which all odd integers appear an even number of times.  If it has no gaps (i.e.~$\la_i-\la_{i+1}\leq1$ for all $i$) and no odd (even) integer appears exactly twice, a partition in the $B_n$ or $D_n$ ($C_n$) theories is called {\it rigid}. We will focus on rigid partition in this paper, which correspond to   rigid operators. The operators  will refer to rigid  operators in the rest of paper. The addition rule of two partitions $\lambda$ and $\mu$ is defined by the additions of each part $\lambda_i+\mu_i$.

The following facts  play important roles in this study.
\begin{prop}{\label{Pb}}
The longest row in a rigid  $B_n$ partition always contains an odd number of boxes. The following two rows of the first row are either both of odd length or both of even length.  This pairwise pattern then continues. If the Young tableau has an even number of rows the row of shortest length has to be even.
\end{prop}
\begin{flushleft}
\textbf{Remark:} If the last row of the partition is odd, the number of rows is odd.
\end{flushleft}

\begin{prop}{\label{Pd}}
For a rigid $D_n$ partition, the longest row  always contains an even number of boxes. And the following two rows are either both of even length or both of odd length. This pairwise pattern then continue. If the Young tableau has an even number of rows the row of the shortest length has to be even.
\end{prop}

\begin{prop}{\label{Pc}}
For a rigid $C_n$ partition, the longest two rows  both contain either  an even or an odd number  number of boxes.  This pairwise pattern then continues. If the Young tableau has an odd number of rows the row of shortest length has contain an even number of boxes.
\end{prop}
\begin{flushleft}
Examples of partitions in the $B_n$,  $D_n$, and $C_n$ theories:
\end{flushleft}
\begin{figure}[!ht]
  \begin{center}
    \includegraphics[width=4.8in]{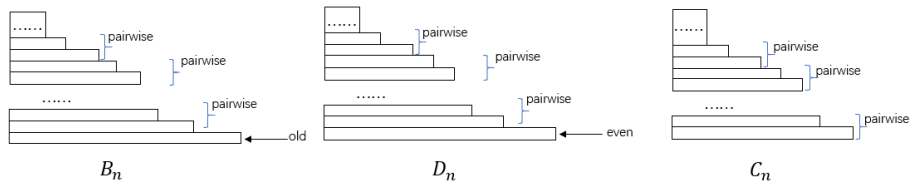}
  \end{center}
  \caption{ Partitions in the $B_n$,  $D_n$, and $C_n$  theories. }
  \label{bdc}
\end{figure}

\section{Symbol invariant of partitions}\label{symbol-s}
 In this section, we introduce the definition of symbol and  its construction \cite{Shou-sc}.
\subsection{Symbol invariant and the construction }
In \cite{Shou-sc}, we  proposed equivalent definitions of symbol for the  $C_n$  and $D_n$ theories which  are  consistent with that for the $B_n$ theory  as   much as possible.
\begin{Def}\label{D2}
 The  symbol  of a partition in the  $B_n $, $ C_n$,  and $D_n$ theories.
\begin{itemize}
  \item For the $B_n$ theory: first we add $l-k$  to the $k$th part of the partition.
Next we arrange the odd parts of the sequence $l-k+\lambda_k$ and the even  parts  in an increasing sequence $2f_i+1$ and in an increasing sequence $2g_i$, respectively.
Then we calculate the terms
 \begin{equation*}
   \al_i = f_i-i+1\quad\quad\quad  \bet_i = g_i-i+1.
 \end{equation*}
 Finally we  write the {\it symbol} as
\begin{equation*}
  \left(\ba{@{}c@{}c@{}c@{}c@{}c@{}c@{}c@{}} \al_1 &&\al_2&&\al_3&& \cdots \\ &\bet_1 && \bet_2 && \cdots  & \ea \right).
\end{equation*}

  \item For the $C_n$ theory: \begin{description}
                                \item[1]If the length of partition is even,  compute the symbol as in the $B_n$ case, and then append an extra 0 on the left of the top row of the symbol.
                                \item[2]  If the length of the partition is odd, first append an extra 0 as the last part of the partition. Then compute the symbol as in the $B_n$ case. Finally,  we delete a 0 in the first entry of the bottom row of the symbol.
                              \end{description}
   \item For the $D_n$ theory: first append an extra 0 as the last part of the partition, and then compute the symbol as in the $B_n$ case. We  delete  two 0's  in the first two entries of the bottom row of the symbol.
\end{itemize}
\end{Def}

\begin{table}
\begin{tabular}{|c|c|c|c|}\hline
Parity of  row & Parity of $i+t+1$ & Contribution to symbol  & $L$  \\ \hline
odd & even  & $\Bigg(\!\!\!\ba{c}0 \;\; 0\cdots \overbrace{ 1\;\; 1\cdots1}^{L} \\
\;\;\;0\cdots 0\;\; 0\cdots 0 \ \ea \Bigg)$ & $\frac{1}{2}(\sum^{m}_{k=i}n_k+1)$  \\ \hline
even & odd  & $\Bigg(\!\!\!\ba{c}0 \;\; 0\cdots \overbrace{ 1\;\; 1\cdots1}^{L} \\
\;\;\;0\cdots 0\;\; 0\cdots 0 \ \ea \Bigg)$   & $\frac{1}{2}(\sum^{m}_{k=i}n_k)$  \\ \hline
even & even  &  $\Bigg(\!\!\!\ba{c}0 \;\; 0\cdots 0\;\; 0 \cdots 0 \\
\;\;\;0\cdots \underbrace{1 \;\;1\cdots 1}_{L} \ \ea \Bigg)$   & $\frac{1}{2}(\sum^{m}_{k=i}n_k)$  \\ \hline
odd & odd  &  $\Bigg(\!\!\!\ba{c}0 \;\; 0\cdots 0\;\; 0 \cdots 0 \\
\;\;\;0\cdots \underbrace{1\; \;1\cdots 1}_{L} \ \ea \Bigg)$     & $\frac{1}{2}(\sum^{m}_{k=i}n_k-1)$  \\ \hline
\end{tabular}
\caption{ Contribution to  symbol of the $i$\,th row$(B_n(t=-1)$, $C_n(t=0)$, and $D_n(t=1)$). }
\label{tsy}
\end{table}
In \cite{Shou-sc}, we give the construction of symbol by  Table \ref{tsy}.  We determine  the contribution to symbol for each row of a partition in the different  theories  uniformly.
For rigid semisimple  operators $(\lambda',\lambda'')$, one can construct the symbol of them  by calculating the symbols for both $\lambda'$ and $\lambda''$,  then add the entries that are `in the same place' of these two results.  An example illustrates the addition rule:
\begin{equation*}
\left(\begin{array}{@{}c@{}c@{}c@{}c@{}c@{}c@{}c@{}c@{}c@{}c@{}c@{}c@{}c@{}} 0&&0&&0&&0&&0&&1&&1 \\ & 1 && 1 && 1 &&1&&1&&2 & \end{array} \right) +
 \left(\begin{array}{@{}c@{}c@{}c@{}c@{}c@{}c@{}c@{}c@{}c@{}c@{}c@{}} 0&&0&&0&&1&&1&&1 \\ & 1 && 1 &&1&&1&&1 & \end{array} \right)=
\left(\begin{array}{@{}c@{}c@{}c@{}c@{}c@{}c@{}c@{}c@{}c@{}c@{}c@{}c@{}c@{}} 0&&0&&0&&0&&1&&2&&2 \\ & 1 && 2 && 2 &&2&&2&&3 & \end{array} \right).
\end{equation*}

\subsection{Representative element of rigid semisimple operators with the same symbol}\label{rep-sym}
In this section, we propose an algorithm to construct a representative element  of the rigid semisimple operators ($\lambda^{'}$, $\lambda^{''}$) with the same symbol invariant. This representative element have a nice structure and is unique.  Especially, its  fingerprint  can be obtained directly in Section \ref{cf}.

 According to Table \ref{tsy}, the contribution to symbol of each row of partition have the following form
\begin{equation}\label{top}
  \Bigg(\!\!\!\ba{c}0 \;\; 0\cdots \overbrace{ 1\;\; 1\cdots1}^{L} \\
\;\;\;0\cdots 0\;\; 0\cdots 0 \ \ea \Bigg)
\end{equation}
 or
  \begin{equation}\label{bottom}
         \Bigg(\!\!\!\ba{c}0 \;\; 0\cdots 0\;\; 0 \cdots 0 \\
\;\;\;0\cdots \underbrace{1\; \;1\cdots 1}_{L} \ \ea \Bigg)
  \end{equation}

The contribution (\ref{top}) and contribution (\ref{bottom}) can be represented by a line above a dashed line and a line under it, respectively.
The length of the line corresponds to the number of  '1' in the formulas (\ref{top}) and (\ref{bottom}).  By using this method, the symbol invariant of a partition can be  visualized as shown in Fig.(\ref{f0}). The contributions of rows of different partitions are represented by different colour. The red lines and the black lines   represent the contributions of the rows of the partition $\lambda^{'}$ and $\lambda^{''}$, respectively.
 \begin{figure}[!ht]
  \begin{center}
    \includegraphics[width=4in]{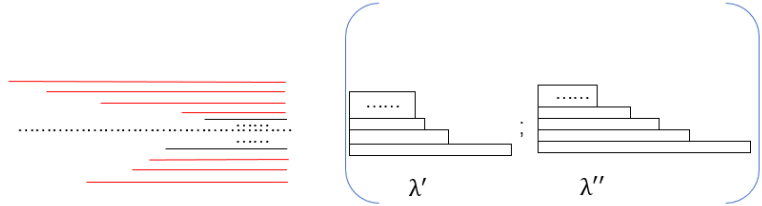}
  \end{center}
  \caption{ Visualization of symbol of the rigid semisimple operator $(\lambda^{'},\lambda^{''})$. }
  \label{f0}
\end{figure}

The representative element is constructed as follows.
First, the longest line of the symbol  must contributed by  the first row of one of the two factors of the rigid semisimple operators.
Without loss of generality,  we assume it is contributed by the first  row of  $\lambda^{''}$. Since the first row of $\lambda^{''}$ is even, its contributions to symbol is on the top row of symbol according to Table (\ref{tsy}). So the longest line $a$ is above  the dotted line as shown in Fig.(\ref{f1}). From the remaining lines, we  chose  the longest lines above  and  under the  dotted line to combine the first  pairwise rows of $\lambda^{''}$.  Without loss of generality, we assume the  former one is longer than  the latter. So these two lines $b$, $c$ of symbol are contributed by  odd pairwise rows of $\lambda^{''}$ as shown in Fig.(\ref{f1}).
 \begin{figure}[!ht]
  \begin{center}
    \includegraphics[width=4in]{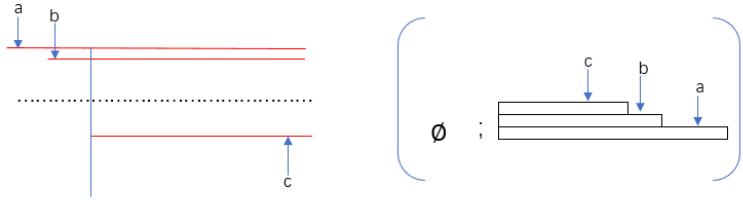}
  \end{center}
  \caption{ The visualization of the  first three rows of $\lambda^{''}$. }
  \label{f1}
\end{figure}

 \begin{figure}[!ht]
  \begin{center}
    \includegraphics[width=4in]{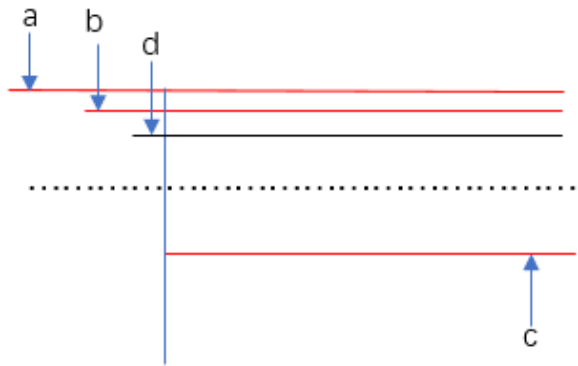}
  \end{center}
  \caption{ The  line $d$ would not be existent since the first line contributed by the first row of $\lambda^{'}$ is under the dashed line. }
  \label{fe}
\end{figure}
  There may be a line $d$ longer than $b$  and shorter than $c$    as shown in Fig.(\ref{fe}).
The line $d$ must be contributed by a row of the  $B_n$ partition $\lambda^{'}$. Since the first row of $\lambda^{'}$ is odd, its contribution to symbol is on the bottom row of symbol, a contradiction. So the contributions of  other rows of $\lambda^{''}$  are shorter than these three lines $a$, $b$, and $c$.
Repeating  the above procedure, we determine another two lines $d$ and $e$ of symbol contributed by pairwise row of $\lambda^{''}$  as shown in   Fig.(\ref{f2}).
\begin{figure}[!ht]
  \begin{center}
    \includegraphics[width=4in]{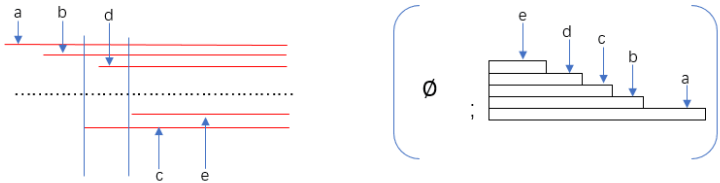}
  \end{center}
  \caption{  The visualization of the  first five rows of $\lambda^{''}$. }
  \label{f2}
\end{figure}

For the next pair lines of symbol, the line $f$  is longer than  $h$,  which means they are contributed by  two even rows of $\lambda^{''}$.
If a line $g$ satisfy $h<g<f$, it must  be contributed by the first row of the partition $\lambda^{'}$ as shown in Fig.(\ref{f3}).

 \begin{figure}[!ht]
  \begin{center}
    \includegraphics[width=4in]{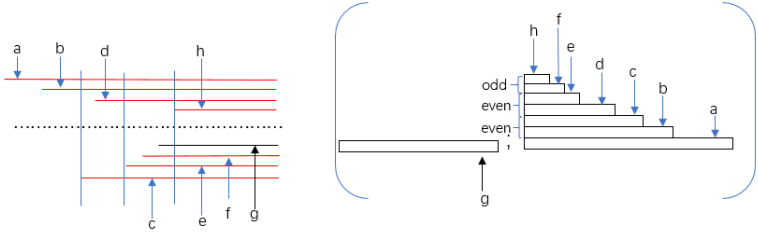}
  \end{center}
  \caption{ Visualization of contributions of symbol. }
  \label{f3}
\end{figure}

These procedure continue until we determine  rows of  $\lambda^{'}$ or  $\lambda^{''}$ for all  lines of the symbol. The following lemma make sure we would reach a consistent result.
\begin{lem}
There would not be  two successive lines  in the same part of the symbol,  excepting the first two rows of a partition.
\end{lem}
\begin{proof} If there are two successive lines $b$ and  $d$ in the same row of the symbol, then there must be another  two lines $c$ and $e$  in the other  part   as shown in Fig.(\ref{fn})($a$). The lines $b$ and  $c$ are  pairwise rows. And the lines $d$ and $e$ are  pairwise rows.  The partition corresponding to Fig.(\ref{fn})($a$) are shown in Fig.(\ref{fn})($b$). The length of the row $d$ is longer than the length of the row $c$,   which is a contradiction.
\end{proof}

 \begin{figure}[!ht]
  \begin{center}
    \includegraphics[width=4in]{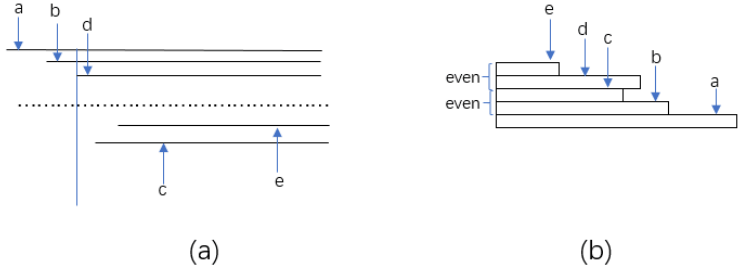}
  \end{center}
  \caption{ $(a)$ The lines $b$ and $d$ are two successive lines in the same row  of symbol. $(b)$, The  partition correspond to the symbol. }
  \label{fn}
\end{figure}

 \begin{figure}[!ht]
  \begin{center}
    \includegraphics[width=4in]{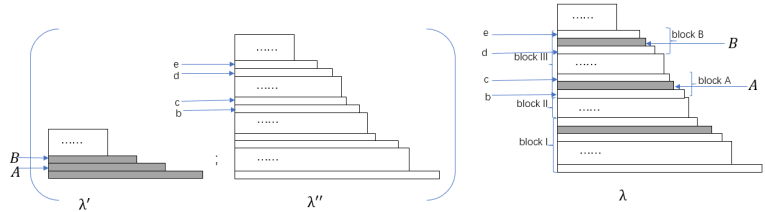}
  \end{center}
  \caption{ The representative element $(\lambda^{'}, \lambda^{''})_R$ and $\lambda$. The gray rows are rows of the partition $\lambda^{'}$.  }
  \label{canobb}
\end{figure}
According to the above algorithm,  we get a representative element $(\lambda^{'}, \lambda^{''})$ and   $\lambda=\lambda^{'}+ \lambda^{''}$ as shown in Fig.(\ref{canobb}). For the  block $III$, both the rows $A$ and $B$ are two rows of  pairwise rows of  the partition $\lambda^{'}$.     Note that the row $A$ and row $B$ are  between the two rows of  pairwise rows of $\lambda^{''}$, which is   a common feature of the rows of the partitions $\lambda^{'}$ in $\lambda$.

According to Table.(\ref{tsy}),  the two lines $b$ and $c$ in Fig.(\ref{canon})($a$) are contributed by  even pairwise rows of $\lambda^{''}$. And  the two lines $d$ and $e$  are contributed by  odd pairwise rows of $\lambda^{''}$.   According to the algorithm to construct the representative element, the rows $A$ and $B$ are two rows of  pairwise rows of the partition $\lambda^{'}$ and are  restricted  by the rows $b$, $c$ and the rows $d$, $e$, respectively. The length of the row $A$ is shorter than the length of the row $b$ and longer than that of the row $c$. And the length of the row $B$ is shorter than the length of the row $d$ and longer than that of the row $e$. The line $A$ is contributed by the first row of   even pairwise rows of $\lambda^{'}$, and  the line $B$ is contributed by the second row of the even pairwise rows. Another example is shown in Fig.(\ref{canon})($b$). The row $A$ is contributed by the first row of  odd pairwise rows,  which  have the same parity with the rows $b$ and $c$.  The row $B$ is contributed by the second row of the odd pairwise rows,  which  have the opposite parity with the rows $d$ and $e$.
\begin{figure}[!ht]
  \begin{center}
    \includegraphics[width=4in]{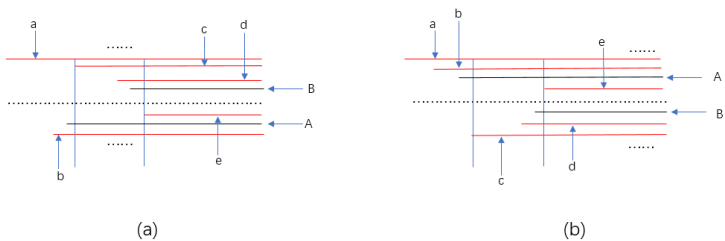}
  \end{center}
  \caption{ The rows $A$ and $B$ are   pairwise rows of $\lambda^{'}$. }
  \label{canon}
\end{figure}

Summary,   for   pairwise rows of the partition $\lambda^{'}$ in $\lambda$,  the parity of the first  row  is the same with the pairwise rows of $\lambda^{''}$ in the block $A$, while the parity of  the second row would be opposite with the pairwise rows of $\lambda^{''}$ in the block $B$ as shown in Fig.(\ref{canon}), which  corresponds to the  block $III$ in Fig.(\ref{canobb}).

\section{Fingerprint invariant of partitions}\label{finger}
In the first two subsections, we would like to introduce the construction of the fingerprint in \cite{SW17-2}. Then we try to find a representative element of the rigid semisimple operators with the same fingerprint and calculate  symbol of it. Finally, we give the fingerprint of the representative element $(\lambda',\lambda'')_R$ in previous section. In this way, we find a map between the symbol invariant and the fingerprint invariant.
\subsection{Fingerprint invariant}
For rigid semisimple surface operator $(\lambda',\lambda'')$, we introduce the definition of fingerprint as follows.  First, add the two partitions  $\lambda=\lambda^{'}+\lambda^{"}$, and then calculate   the partition $\mu=Sp(\lambda)$
\begin{eqnarray}\label{mu}
\mu_i=Sp(\lambda)_i=
\left\{ \begin{aligned}
         & \lambda_i + p_{\lambda}(i) \quad  \quad   \textrm{if} \quad \lambda_i  \textrm{ is odd  and} \quad \lambda_i \neq \lambda_{i-p_{\lambda}(i)}, \\
         & \lambda_i   \quad  \quad   \quad   \quad  \quad        \textrm{ otherwise}
       \end{aligned} \right.
\end{eqnarray}
Next define the function $\tau$ from an even positive integer $m$ to $\pm$ as follows. For a partition in the $B_n$ and $D_n$ theories,  $\tau(m)$ is $-1$ if  at least one $\mu_i$ such that $\mu_i= m$ and either of the following three conditions is satisfied.
\begin{eqnarray}\label{con}
  &&(i)\quad\quad\quad\quad\quad \,\,\mu_i\neq \lambda_i  \non \\
 && (ii) \quad\quad\quad\sum^{i}_{k=1}\mu_k \neq \sum^{i}_{k=1}\lambda_k\\
&& (iii)_{SO} \quad\quad\quad \lambda^{'}_{i} \quad\textrm{is odd}.  \non
\end{eqnarray}
Otherwise  $\tau$ is 1.
For the  partitions in the $C_n$ theory,  the definition is exactly  the same except the condition $(iii)_{SO}$ is  replaced by
$$(iii)_{Sp}\quad\quad \lambda^{'}_{i}  \quad   \textrm{is even}.$$
Finally  construct  a pair of partitions $[\alpha;\beta]$. For each pair of parts of $\mu$ both equal to $a$, satisfying $\tau(a)=1$,   retain one part $a$ as a part of  the partition $\alpha$. For each part of $\mu$ of size $2b$, satisfying $\tau(2b)=-1$,  retain $b$ as a part of  the partition $\beta$.

Note that $\mu_i\neq\lambda_i$ only happen at the end of a row.
According to the above definition of fingerprint, we have the following important lemma.
\begin{lem}\label{ff}
Under the map $\mu$, the change of  rows depend on  the parity of the row and  the sign of $p_\lambda(i)$.
\begin{center}
  \begin{tabular}{|c|c|l|}
  \hline
  Parity of row & Sign & Change \\
  odd & $-$ & $\mu_i=\lambda_i-1$ \\
  even & $-$ & $\mu_i=\lambda_i+1$ \\
  even & $+$ & $\mu_i=\lambda_i$ \\
  odd & $+$ & $\mu_i=\lambda_i$ \\
  \hline
\end{tabular}
\end{center}
\end{lem}

\begin{figure}[!ht]
  \begin{center}
    \includegraphics[width=2.5in]{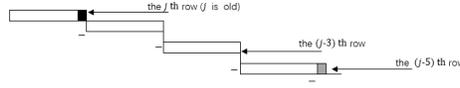}
  \end{center}
  \caption{ The difference of the height of the $(j-1)$th row and the $(j-3)$th row is two, violating the rigid condition $\lambda_i-\lambda_{i+1}\leq 1$.}
  \label{con2}
\end{figure}
It is easy to prove the following fact.
\begin{prop}{\cite{SW17-2}}
For the partition $\lambda$ with $\lambda_i-\lambda_{i+1}\leq 1$, the condition (\ref{con}) $(i)$  imply $(ii)$.
\end{prop}
We would give an example where  the condition $(ii)$ works. As shown in Fig.(\ref{con2}),  the heights of the $(j-1)$th row, the $(j-3)$th row, and the $(j-5)$th are even.  And  the difference of the height of the $(j-1)$th row and the $(j-3)$th row violate the rigid condition $\lambda_i-\lambda_{i+1}\leq 1$ as well as that of the height of the $(j-3)$th row and the $(j-5)$th row. The part $(j-3)$ satisfy the condition (\ref{con})$(ii)$ but do not satisfy the condition $(i)$.

In this study,   the neglect of the condition $(ii)$ in the definition of fingerprint  would not change the conclusions in this paper. In the rest of the paper, we would focus on  the partition $\lambda=\lambda^{'}+\lambda^{''}$ with $\lambda_i-\lambda_{i+1}\leq 1$.

\subsection{Construction of the fingerprint invariant }
In this subsection, we introduce the construction of fingerprint of rigid $(\lambda^{'}, \lambda^{''})$ in the $B_n$ theory.  We use a simplified model which is enough for the construction of the fingerprint of the representative element  $(\lambda^{'}, \lambda^{''})_R$ in next subsection. For the $C_n$ and $D_n$ theories, we can get the same results with minor modifications.
\begin{flushleft}
{ \textbf{Operators $\mu_{e11}$, $\mu_{e12}$, $\mu_{e21}$, and $\mu_{e22}$}}
\end{flushleft}
The partition $\lambda=\lambda^{'}+\lambda^{''}$ is constructed  by inserting rows of $\lambda^{'}$  into $\lambda^{''}$  one by one. Without confusion, the image of the map $\mu$  is also denoted as $\mu$.
 The partition $\lambda=\lambda^{'}+\lambda^{''}$  is decomposed into several  blocks such as blocks $I, II, III$ as shown in  Fig.(\ref{model}).
\begin{figure}[!ht]
  \begin{center}
    \includegraphics[width=4.9in]{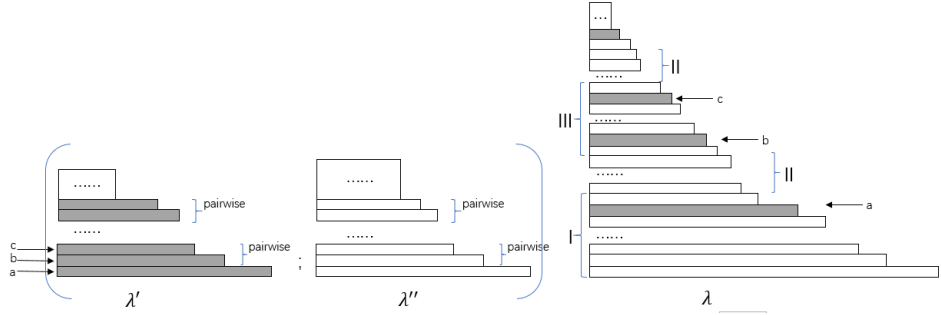}
  \end{center}
  \caption{ Blocks in the partition $\lambda$.}
  \label{model}
\end{figure}

The  $III$ type block   consist of   pairwise rows of $\lambda^{'}$ and  $\lambda^{''}$ as shown in Fig.(\ref{model}). The upper boundary  and the lower boundary of the block consist of a row of $\lambda^{'}$ and one pairwise rows of  $\lambda^{''}$. When the pairwise rows $b$, $c$ of $\lambda^{'}$  are even, we define four operators $\mu_{e11}$, $\mu_{e12}$, $\mu_{e21}$, and $\mu_{e22}$  according to the positions of the  rows of $\lambda^{'}$ inserted into  $\lambda^{''}$, as shown in Fig.(\ref{ob}). The first and last row of the block do not change under the map $\mu$ as well as the  even pairwise rows of $\lambda^{''}$. We append a box to the first row of odd pairwise rows of  $\lambda^{''}$ and delete a box at the end of the second row. If $b$ is not the first row of the block, it will behave as the first row of  pairwise rows of  $\lambda^{''}$ as shown in Fig.(\ref{model}). If $c$ is not  the last row of the block, it will behave as the second row of  pairwise rows of  $\lambda^{''}$.  When the pairwise rows of $\lambda^{'}$ inserted into $\lambda^{''}$  are odd, we define four operators $\mu_{o11}$, $\mu_{o12}$, $\mu_{o21}$, and $\mu_{o22}$ similarly. However the odd rows and the even rows interchange roles in this case. We  can also consider $III$ type block    consisting of one  pairwise rows $\lambda^{''}$ and pairwise rows of  $\lambda^{'}$.
\begin{figure}[!ht]
  \begin{center}
    \includegraphics[width=4.9in]{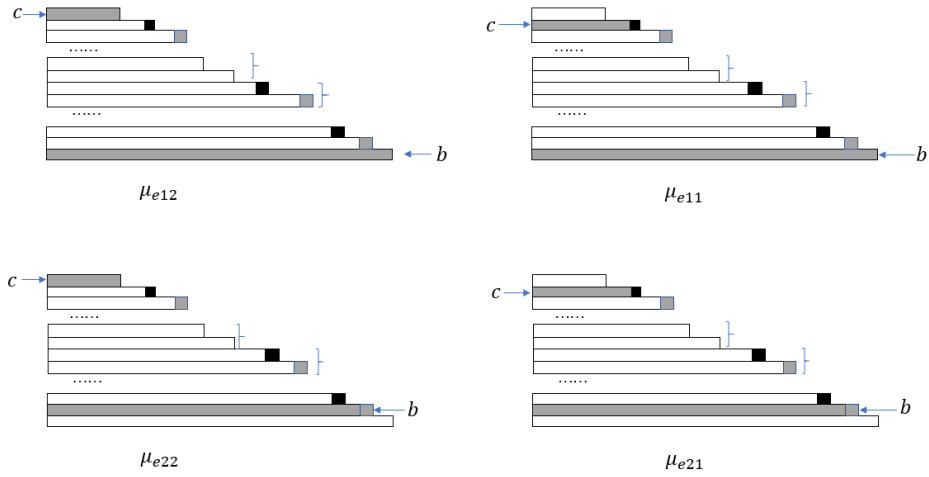}
  \end{center}
  \caption{The  images of $III$ type blocks under the map $\mu$.}
  \label{ob}
\end{figure}

The  $II$ type block  only consist of  pairwise rows of $\lambda^{'}$ or  $\lambda^{''}$ as shown in Fig.(\ref{mue}) which do not change under the map $\mu$.
\begin{figure}[!ht]
  \begin{center}
    \includegraphics[width=4.9in]{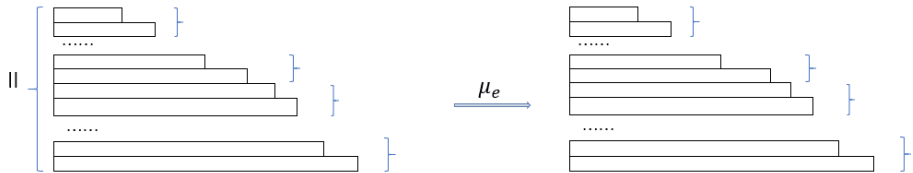}
  \end{center}
  \caption{ $II$ type block and its image under the map $\mu$.}
  \label{mue}
\end{figure}

The  $I$ type block   consist of  pairwise rows of  $\lambda^{''}$, besides of the  first row  $a$ of $\lambda^{'}$ and the  first row of $\lambda^{''}$. We define two operators $\mu_{e1}$and $\mu_{e2}$ correspond to row $a$  above the pairwise rows in Fig.(\ref{mu1})(a)  and  between the  pairwise rows in Fig.(\ref{mu1})(b), respectively. The last row of the block and the  even pairwise rows of $\lambda^{''}$ do not change under the map $\mu$. We append a box to the first row of odd pairwise rows of  $\lambda^{''}$ and delete a box at the end of the second row of odd pairwise rows. If $a$ is not the last row of the block, it will behave as the second row of  pairwise rows of  $\lambda^{''}$ as shown in Fig.(\ref{mu1})(b).    We  can also consider $I$ type block   consist of  pairwise rows of  $\lambda^{'}$, besides of the  first rows of $\lambda^{'}$ and  $\lambda^{''}$. Similarly, we  define two operators $\mu_{o1}$ and $\mu_{o2}$. However the odd rows and the even rows interchange roles in this case.
There is no $I$ type block of the $\lambda$ in the $C_n$ theory. So the fingerprint of the $C_n$ rigid semisimple operators  is a  simplified  version of that in the $B_n$ and $D_n$ theories.
\begin{figure}[!ht]
  \begin{center}
    \includegraphics[width=4.9in]{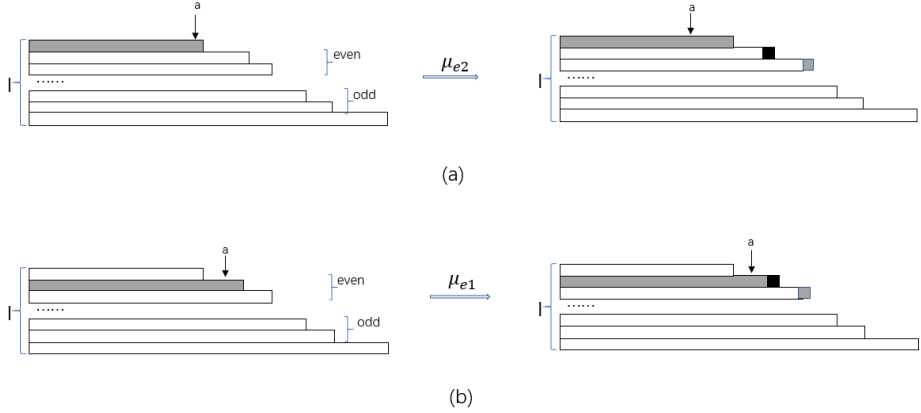}
  \end{center}
  \caption{ $I$ type block and its image under the map $\mu$.}
  \label{mu1}
\end{figure}

\subsection{Representative element of rigid semisimple operators with the same fingerprint}\label{rep-finger}
With the fingerprint given , it is hope to  find a natural representative  element $(\lambda^{'}, \lambda^{''})_r$   which  similar to the representative  element $(\lambda^{'}, \lambda^{''})_R$.
Since pair of  fingerprint and $\lambda$  would determine the  rigid semisimple operator  $(\lambda^{'}, \lambda^{''})$, we would find a  representative  element $\lambda_r$  of $\lambda$.

An   element $\lambda_r$ for  $\lambda^{'}+ \lambda^{''}$  can be constructed  from the fingerprint invariant  naturally, with the same fingerprint invariant $[\alpha, \beta]$.     For each part $a$  of the partition $\alpha$ retain $a^{2}$. From the integers so obtained form the partition $\lambda^{'}_r$. For each part $b$  of the partition $\beta$ retain $2b$. From the integers so obtained form the partition $\lambda^{''}_r$.
Then  $\mu_r=\lambda^{'}_r+\lambda^{''}_r$ is the image of $\lambda$ under the map $\mu$. All the partition   $\lambda$ for rigid semisimple operator $(\lambda^{'}, \lambda^{''})$ with the same fingerprint leads to  it by the map $\mu$.

A naive guess is that there always exist  a operator $(\lambda^{'}, \lambda^{''})_r$ and $\mu_r=\lambda^{'}+\lambda^{''}$.   Unfortunately, it is wrong.  We would give a counterexample as follows. The $B_6$ operator $(2^21^9,\emptyset)$ is the only operator with the dimension\footnote{This is another invariant of rigid semisimple operators which contain less information than symbol invariant and fingerprint invariant. } of 20. Its image under the map $\mu$ is
$$2^21^9\rightarrow 2^21^8$$
where $\lambda \neq \mu$, a contradiction.  However, for the rigid semisimple operator with  $\lambda=\mu$, their symbol have nice properties.
\begin{flushleft}
 \textbf{ Symbol of $\mu_r$:}
\end{flushleft}
The operator $\mu_e$ in Fig.(\ref{mue}) suggest that the possible structure of the partition $\mu_r$.  It consist of pairwise rows and like a partition in the $C_n$ theory. Since the rigid semisimple operator can be obtain by  pair of $[\alpha, \beta]$ and $\mu_r$, of which the  symbol  can be obtained  directly  according to Table \ref{tsy}.

\subsection{Fingerprint of the representative element $(\lambda^{'},\lambda^{''})_R$}\label{cf}

The representative  semisimple operator $(\lambda^{'}, \lambda^{''})_R$ and the corresponding partition $\lambda=\lambda^{'}+\lambda^{''}$ are  shown in Fig.(\ref{canobb}). There are three type of blocks in the representative element $\lambda_R$  which are the $I$, $II$, and $III$ type blocks.
\begin{itemize}
  \item For the $I$ type block as shown in Fig.(\ref{cab}), the $\mu$ partition is given by the operator $\mu_{o1}$ in Fig.(\ref{ob}) since $l3$ is odd. The height of the last row of the block is even which  satisfy the condition $(iii)$ but do not satisfy the condition $(i)$. The other rows of the block do not satisfy the condition $(iii)$ but the even rows do change under the map  $\mu_{o1}$ thus satisfying the condition $(i)$. And the odd rows do change under the map  $\mu_{e1}$ thus satisfying the condition $(i)$.
 \begin{figure}[!ht]
  \begin{center}
    \includegraphics[width=4in]{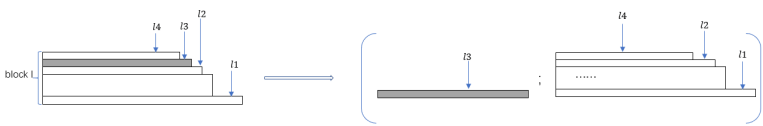}
  \end{center}
  \caption{ Block $I$ of the representative element . }
  \label{cab}
\end{figure}

  \item For the $II$ type block as shown in Fig.(\ref{canobb}), the $\mu$ partition is given by the operator $\mu_{e}$ in Fig.(\ref{ob}). The rows of the block do not change under this operator, thus  not satisfying the condition $(i)$. However the block satisfy the condition $(iii)$.

  \item For the $III$ type block as shown in Fig.(\ref{canobb}), the $\mu$ partition is given by the operator $\mu_{e21}$ or $\mu_{o21}$  in Fig.(\ref{ob}), according to the parities of the pairwise rows $A$, $B$.
The height of the last row of the block is even. It satisfy the condition $(iii)$ but do not satisfy the condition $(i)$. The other rows of the block do not satisfy the condition $(iii)$ but the even rows do change under the map  $\mu_{o21}$, thus satisfying  the condition $(i)$.  And the odd rows do change under the map  $\mu_{e21}$, thus  satisfying the condition $(i)$.
\end{itemize}

\section{$C_n$ and $D_n$ theories}\label{cd}
For the rigid semisimple operators in the $C_n$ and $D_n$ theories, we can get the similar results by the same strategy in previous sections with minor modifications.

\begin{flushleft}
 \textbf{ Fingerprint of $\lambda_R$:} Now we discuss the fingerprint invariant of the representative element $(\lambda^{'}, \lambda^{''})_R$ for the rigid semisimple operators in the $C_n$ and $D_n$ theories.
\end{flushleft}
The first two rows of $C_n$ partitions  are pairwise rows according to Proposition \ref{Pc}. Thus there is no $I$ type block of the $\lambda$ in the $C_n$ theory. So the fingerprint of the $C_n$ rigid semisimple operators  is a  simplified  version of that in the $B_n$ and $D_n$ theories.  The other processes to calculate the fingerprint of $\lambda_R$ is exact the same with that of the $B_n$ rigid semisimple operator.

The first  row of $D_n$ partitions  is even and not of  pairwise rows according to Proposition \ref{Pd}.  For the  $I$ type block of the $\lambda$ in the $D_n$ theory, there are only $\mu_{e1}$ and $\mu_{e2}$ operators.   The other processes to calculate the fingerprint of $\lambda_R$ is exact the same with that of the $B_n$ rigid semisimple operator.

\begin{flushleft}
 \textbf{ Symbol of $\mu_r$:}  we will discuss the symbol invariant of the element $\mu_r$.
\end{flushleft}
The operator $\mu_e$ in Fig.(\ref{mue}) suggest that the possible structure of the partition $\mu_r$.  It consist of pairwise rows and like a partition in the $C_n$ theory. The rigid semisimple operator can be obtain by  pair of $[\alpha, \beta]$ and $\mu_r$, of which  the  symbol  can be obtained  directly  according to Table \ref{tsy}.


\begin{thebibliography}{99}
\bibitem{CM93}
D.~H. Collingwood and W.~M. McGovern, { Nilpotent orbits in semisimple Lie
  algebras},
\newblock Van Nostrand Reinhold, 1993.

\bibitem{GW06}
S.~Gukov and E.~Witten, {Gauge theory, ramification, and the geometric
  {Langlands} program},  arXiv:hep-th/0612073


\bibitem{GW08}
S.~Gukov and E.~Witten, {Rigid surface operators},   arXiv:0804.1561


\bibitem{Wy09}
N.Wyllard, { Rigid surface operators and $S$-duality: some proposals},   arXiv: 0901.1833

\bibitem{ShO06}
B.~Shou, { Symbol, Rigid surface operaors and $S$-duality},  preprint, 26pp, arXiv: 1708.07388


\bibitem{Lu79}
G.~Lusztig, { A class of irreducible representations of a Weyl group},  Indag.Math, 41(1979), 323-335.

\bibitem{Lu84}
G.~Lusztig, { Characters of reductive groups over a finite field},
\newblock Princeton, 1984.


\bibitem{Sp92}
N.~Spaltenstein, {Order relations on conjugacy classes and the
  {Kazhdan-Lusztig} map},  {\it Math. Ann.}, {\textbf 292} (1992) 281.

\bibitem{Montonen:1977}
C.~Montonen and D.~I. Olive, ``Magnetic monopoles as gauge particles?,'' {\em
  Phys. Lett.} {\bf B72} (1977)
117;

\bibitem{GNO76}
P.~Goddard, J.~Nuyts, and D.~I. Olive, {Gauge theories and magnetic charge},
  {\it Nucl. Phys.}, {\textbf B125} (1977)
1.

\bibitem{AKS06}
P.~C. Argyres, A.~Kapustin, and N.~Seiberg, {On {$S$-duality} for
  non-simply-laced gauge groups},  {\it JHEP}, {\textbf 06} (2006) 043,arXiv:hep-th/0603048


\bibitem{GM07}
  J.~Gomis and S.~Matsuura,
  {Bubbling surface operators and $S$-duality},  \\
  {\it JHEP}, {\textbf 06} (2007) 025,arXiv:0704.1657

\bibitem{DGM08}
N.~Drukker, J.~Gomis, and S.~Matsuura, {Probing $\cN=4$ SYM with surface
  operators},  {\it JHEP}, {\textbf 10} (2008) 048,  arXiv:0805.4199

\bibitem{GW14}
S.~Gukov, {Surfaces Operators},   arXiv:1412.7145

\bibitem{Sh06}
B.~Shou,  {Solutions of Kapustin-Witten equations for ADE-type groups}, preprint, 26pp, arXiv:1604.07172

\bibitem{Shou-sc}
B.~Shou,  {   Symbol Invariant of Partition and Construction}, preprint, 31pp,  arXiv:1708.07084

\bibitem{SW17}
B.~Shou, and Q.~Wu, { Construction of the  Symbol Invariant of Partition}, preprint, 31pp,  arXiv:1708.07090

\bibitem{SW17-2}
B.~Shou, and Q.~Wu, {Fingerprint Invariant of Partitions and Construction }, preprint, 23pp,  arXiv:1711.03443


\bibitem{iv}
B.~Shou,  { Invariants of Partitions}, in preparation.

\bibitem{HW07a}
M.~Henningson and N.~Wyllard, {Low-energy spectrum of {$\cN = 4$}
  super-{Yang-Mills} on {$T^3$}: flat connections, bound states at threshold,
  and {$S$-duality}},  {\it JHEP}, {\textbf 06} (2007), arXiv:hep-th/0703172

\bibitem{HW07b}
M.~Henningson and N.~Wyllard, {Bound states in {$\cN = 4$} {SYM} on {$T^3$}:
  {$\Spin(2n)$} and the exceptional groups},  {\it JHEP}, {\textbf 07} (2007) 084, arXiv:0706.2803


\bibitem{HW08}
M.~Henningson and N.~Wyllard, {Zero-energy states of $\cN = 4$ {SYM} on
  $T^3$: $S$-duality and the mapping class group},  {\it JHEP}, {\textbf 04} (2008)
  066, arXiv:0802.0660




\bibitem{Shou 1:2016}
B.~Shou, { Symbol, Surface opetators and $S$-duality}, preprint 27pp, {\tt  arXiv: 1708.07388}.


\bibitem{rso}
B.~Shou, { Rigid Surface opetators and Symbol Invariant of Partitions}, preprint 23pp, {\tt  arXiv: 1708.07388}.

\bibitem{Spaltenstein:1992}
N.~Spaltenstein, Order relations on conjugacy classes and the {Kazhdan-Lusztig} map, {\em Math. Ann.} {\bf 292} (1992) 281.


\bibitem{Lusztig:1984}
G.~Lusztig, { Characters of reductive groups over a finite field},\newblock Princeton, 1984.

\bibitem{Symbol 2}
G.~Lusztig, { A class of irreducible representations of a Weyl group}, {\em Indag.Math}, 41(1979), 323-335.

\bibitem{Sh11}
B.~Shou, J.F.~Wu and M.~Yu, { AGT conjecture and AFLT states: a complete construction}, preprint, 28 pp., arXiv:1107.4784


\bibitem{Sh14}
B.~Shou, J.F.~Wu and M.~Yu, Construction of AFLT States by Reflection Method and Recursion Formula, {\em Communications in Theoretical Physics}, {\bf 61} (2014) 56--68

\bibitem{Sh15}
Z.S.~Liu, B.~Shou, J.F.~Wu,Y.Y.~Xu and M.~Yu, Construction of AFLT States for $W_n\bigotimes \mathcal{H}$, Symmetry, Analytic Continuation and Integrability on AGT Relation, {\em Communications in Theoretical Physics}, {\bf 63} (2015) 487--498

\bibitem{Localization}
N.~Nekrasov,  A.~Okounkov, {Seiberg-Witten theory and random partitions}, {\em The Unity of Mathematics},  525-596.  {\tt arXiv:hep-th/0306238}.

\end{thebibliography}
\end{document}